\documentclass[11pt]{amsart}
\usepackage{amsmath, amssymb, mathrsfs}
\usepackage{amssymb, color}
\usepackage{mathpazo}
\def\lbr{\left\{}
\def\rbr{\right\}}

\def\me{{\mathbb  E}}

\def\mr{{\mathbb  R}}

\def\mp{{\mathbb  P}}

\newcommand{\rf}[1]{(\ref{#1})}
\def\lnfrac#1#2{\raise.7ex \hbox{\Small $#1$}
  \kern-.15em/\kern-.15em  \lower.2ex \hbox{\Small $#2$}}

\theoremstyle{plain}
\newtheorem{theorem}{Theorem}[section]

\newtheorem{lemma}[theorem]{Lemma}

\theoremstyle{definition}

\numberwithin{equation}{section}

\begin{document}

\title[Arc-sine law]{Geometric Deviation From L\'evy's Occupation Time Arcsine Law}
\vskip 1cm

\author{Elton P. Hsu}
\address{Department of Mathematics\\
        Northwestern University \\ Evanston, IL 60208}
\email{ehsu\@@math.northwestern.edu}
\author{Cheng Ouyang}
\address{Department of Math, Statistics and Computer Science\\
University of Illinois at Chicago \\ Chicago, IL 60607}
\email{couyang@math.uic.edu}



\begin{abstract} We prove a geometric extension of L\'evy's occupation time arcsine law near a hypersurface on a Riemannian manifold. The deviation from the classic arcsine law is of the order of the square-root of the time horizon and is expressed explicitly in terms of the mean curvature of the hypersurface and the local time of the underlying standard Brownian motion.
\end{abstract}
\maketitle

\section{Introduction}
Let $W = \lbr W_t, \, t\ge 0\rbr$ be a standard one-dimensional Brownian motion starting from the origin.  L\'evy's occupation time arcsine law (L\'evy~\cite{Levy, Levy2}) states that the total time spent by $W$ in the half line $\mr_+ = [0, \infty)$ up to time $1$ obeys the arcsine law
\begin{equation*}
\mp\left[\int_0^11_{\mr_+}(W_s)\, ds\le x\right] = \frac2\pi\arcsin\sqrt x.
\end{equation*}
A proof using techniques developed since L\'evy was given by Kac~\cite{Kac}, which was reproduced with all necessary details in Karatzas and Shreve~\cite{KaratzasS}. Another more intriguing proof can be found in Revuz and Yor~\cite{RevuzY}. 

For a Brownian motion on a Riemannian manifold stating from a smooth hypersurface, we expect that the arcsine law holds approximately for small time horizons. It is therefore natural to investigate the deviation from the classical arcsine law caused by the geometry of the hypersurface relative to the ambient space. At a preliminary level, the problem is not difficult to formulate. Let $M$ be a Riemannian manifold and $N$ a smooth hypersurface in $M$. Locally $N$ divides $M$ into two disjoint parts $M_+$ and $M_-$ with $N$ as the common boundary. Take a point $o\in N$ as the starting point of a Riemannian Brownian motion $X$. We expect the normalized total time the Brownian motion spends in the half-space $M_+$
\begin{equation}
 \frac1t\int_0^t1_{M_+}(X_s)\, ds\label{occupation}
\end{equation}
approaches the arcsine law in distribution when $t$ tends to zero, and the deviation from this law should be related to the local geometry of the hypersurface $N$ in the ambient space $M$ near the starting point $o$. The purpose of the present work is to formulate and prove a result in this direction. Our result {\sc Theorem}  \ref{main} shows that, if formulated properly, the deviation is of the order  $\sqrt t$ and is proportional to the mean curvature of $N$ at $o$ and the integral 
$$\int_0^1 s \, dL_s,$$
where $L$ is the (time scaled) local time of the Riemannian Brownian motion. 

Since the random variables \rf{occupation} converges to the arcsine law only in distribution, we need to choose our setting carefully in order to formulate a rigorous result. By the principle of not feeling the boundary (see Hsu~\cite{Hsu1}), we will assume that our Riemannian manifold $M$ is diffeomorphic to $\mr^n$ with a globally defined Riemannian metric $g$, which may be assumed to be euclidean outside a neighborhood of the origin $o$. In the global coordinates $x = (x^1, \tilde x) = (x^1, \ldots, x^{n-1}, x^n)$, the hypersurface
is given by $N = \lbr x\in\mr^n\vert x^1 = 0\rbr$. For each fixed positive $t$, we can rewrite the occupation time in \rf{occupation} in the form
$$T_t = \int_0^11_{M+}\left(X_{st}\right)ds.$$
Instead of working with a single Riemannian Brownian motion $X$, we will generate a family of time scaled Brownian motion $X^t = \lbr X^t_s, \ s\ge 0\rbr$ as the solutions of a stochastic differential equation driven by a fixed euclidean Brownian motion $W$ in such a way that our $X^t$ (defined on the same probability space) has the same law as the $\lbr X_{st}, \ s\ge 0\rbr$ for each fixed $t$ (see below). 
Under this setting, our main result can be stated as follows. .

\begin{theorem} \label{main} Let $W$ be the (common) driving Brownian motion (more precisely, the first component thereof) of the (time-scaled) Riemannian Brownian motions $X^t$ and $L = \lbr L_s, \, s\in\mr_+\rbr$ the local time of $W$ at $x = 0$. Then in $L^p$ with $p\ge 1$, as $t\downarrow 0$, 
\begin{equation*}
T_t= \int_0^11_{\mr_+}(W_u)du+\frac12{\sqrt{t}H}\int_0^1u\,dL_u+O(t^{3/4}).
\end{equation*}
Here $H$ is the mean curvature of the hypersurfacee $N$ at the starting point $o$. 
\end{theorem}

{\sc Section 4} will be devoted to the proof of this result. In {\sc Section} 2 we review the definition of the local time of a one-dimensional Brownian motion and a time-dependent occupation time formula. In {\sc Section} 3 we review the formulation of Riemannian Brownian motion in semi-geodesic local coordinates on a Riemannian manifold.  In the last {\sc Section 5}, we will make several value added remarks that the reader will regret to skip. 

\section{A time-dependent occupation time formula}
The local time $L^x = \lbr L^x_s,\, s\in\mr_+\rbr$ at $x\in\mr$ of a one-dimensional Brownian motion $W$ is defined by the relation
\begin{align}\label{local time representation}\vert W_t -x\vert = \vert W_0-x\vert + \int_0^t \text{sgn} (W_s-x)\, dW_s + L^x_t.\end{align}
It is well known result of Wang~\cite{Wang} that there is a version of the local time jointly continuous in $(x,s)\in \mr\times\mr_+$.  The following occupation time formula holds: with probability one, for every nonnegative measurable function $\mr\times\mr_+\rightarrow\mr_+$ and every $t$,
\begin{equation}
\int_0^t\varPsi(W_s, s)\, ds = \int_\mr\left[\int_0^t \varPsi(x, s)\, dL_s^x\right]\, dx.\label{cop}
\end{equation} 
See {\sc Exercise} (1.15) on page 232 of Revuz and Yor~\cite{RevuzY}. 

It is a classical fact that for a fixed $t$, the local time $L_t = L^0_t$ at $x = 0$ has the same law as $\vert W_t\vert$. 

\section{Brownian motion in semi-geodesic coordinates}

Let $M$ be a Riemannian manifold and $N$ a smooth hypersurface in $M$. Near a point $o\in N$,  the manifold can be parametrized by a system of semi-geodesic coordinates $x = (x^1, \tilde x)$, where $x^1$ is the signed distance of $x$ to $N$ and $\tilde x = (x^2, \ldots, x^n)$ a system of normal coordinates of $N$ centered at $o$. In these coordinates, the Riemannian metric is given by
$$g_{ij}(x) = \begin{cases}
\delta_{1j}, &i = 1, 1\le j\le n;\\
\delta_{ij} + 2\varPi_{ij}x_1 +  O(\vert x\vert^2),\ &2\le i, j\le n.
\end{cases}$$
Here 
$$\varPi_{ij}=\left\langle \frac\partial{\partial x^1}, \nabla_{\partial/\partial x^i}\frac\partial{\partial x^j}\right\rangle$$
is the second fundamental form of $N$ at $o$ expressed in terms of the normal coordinates $\tilde x = (x^2, \ldots, x^n)$; see Hsu~\cite{Hsu2}.

A Riemannian Brownian motion $X= (X^1, \tilde X)$ on $M$ is generated by half of the Laplace-Beltrami operator 
$$\Delta = \frac1{\sqrt{\det g}}\frac\partial{\partial x^i}\left(\sqrt{\det g}g^{ij}\frac\partial{\partial x^j}\right),$$
where $g = (g_{ij})$ is the metric matrix and $(g^{ij}) = g^{-1}$ its inverse.  We are interested in the first component $X^1$ of the Brownian motion. A simple calculation shows that
$$ \Delta = \left(\frac\partial{\partial x^1}\right)^2 + b^1(x)\frac\partial{\partial x^1} + \text{partial derivatives with respect to}\ \tilde x,$$
where 
\begin{equation} 
b^1(x) = H+ O(\vert x\vert),\label{drift}
\end{equation}
with $H = \text{Tr}\, \varPi$, the mean curvature of the hypersurface $N$ at $o$.
Using the stochastic differential equation for the diffusion process generated by $\Delta/2$ (see Hsu~\cite{Hsu}), the first component of the Brownian motion on $M$ is given by 
\begin{equation}X_s^1=W_s+\frac12\int_0^s b^1(X_u)du,\label{firstcomp}\end{equation}
where $W$ is a one-dimensional Brownian motion. However, as explained in the introduction, for the current problem, we consider the time-scaled process $X^t = \lbr X^t_s, \, 0\le s\le1\rbr$ for each fixed $t$ and write
\begin{equation}
X^{t,1}_s = \sqrt t W_s +\frac t2\int_0^s b^1(X^t_u)\, du.\label{scaledb}
\end{equation}
Thus $\lbr X^t_s, \ s\in\mr_+\rbr$ and $\lbr X_{ts}, \ s\in\mr_+\rbr$ have the same law for each fixed $t$. Its occupation time in the half space $M_+$ up to time $1$ is given by
\begin{equation}
T_t = \int_0^1 1_{\mr_+}(X^{t,1}_u)\, du,\label{ot}
\end{equation}
which has the same distribution as the normalized occupation time of a Riemannian Brownian moiton in $M_+$ up to time $t$, again for each fixed $t$. From this expression and \rf{scaledb} it is clear that as $t\downarrow 0$, the random variables $T_t$ converges almost surely to 
$$ \int_0^11_{\mr+}(W_s)\, ds,$$
which, according to L\'evy, has the arcsine law. In the next section we will calculate the deviation of $T_t$ from this limit. 

\section{Deviation of the occupation time from the arcsine law}

The occupation time $T_t$ of $X^t$ up to time 1 in the half space $M_+$ is given by  \rf{ot}. Let $L = \lbr L_s, \, 0\le s\le 1\rbr$ denote the local time of the Brownian motion $W$ at $x = 0$.  
The main result of this section is {\sc Theorem} \ref{lpapprox}, the first order approximation of $T_t$ in the sense of $L_p$. 
We start with the following lemma, which reduces the problem to that of a standard one-dimensional Brownian motion. 

\begin{lemma}\label{cut tail}
For any $p\geq 1$ and $0<\alpha< 1/2$, as $t\downarrow 0$ we have 
$$\me\left|T_t- \int_0^11_{\mr_+}(W_u+\frac12{\sqrt{t}}Hu)du\right|^p=O(t^{(1/2+\alpha)p}).$$
\end{lemma}
\begin{proof} We need to localize the calculation to a neighborhood of the origin $o$ without affecting the leading deviation term. For this purpose, let $B(t^\alpha)$ be the (shrinking) ball of radius $t^\alpha$ centered at the origin and 
\begin{align}\label{tau}\tau^t=\inf\{s\geq 0: X^t_s\in B(t^\alpha)^c\}\end{align}
the first exit time of $X^t$ from this ball. From the stochastic differential equation for $X^t$ we see that it is the sum of $\sqrt t$ times a euclidean Brownian motion and a drift uniformly bounded by a constant multiple of $t$. For the less obvious first component this is more explicitly displayed in \rf{scaledb}. Hence there is a positive constant $c$ such that
$$\mp\{ \tau^t>1\}\le e^{-ct^{-(1/2-\alpha)}}=o(t^n)$$
for any positive $n$; hence, 
\begin{align*}
&\me\left| T_t- \int_0^11_{\mr_+}(W_u+\frac12\sqrt{t}Hu)du\right|^p\\
=&\me \left(\left\vert T_t- \int_0^11_{\mr_+}(W_u+\frac12\sqrt{t}Hu)du\right\vert^p;  \tau^t \leq 1\right)+o(t^n).
\end{align*}
From \rf{scaledb} the drift of $X^{t,1}/\sqrt t$ after removing its limiting value $Hs$ is 
\begin{align}\label{def h}\sqrt t\, h_s=\sqrt t\int_0^s\left[b^1({X}_u^t)- H\right]\, du.\end{align}
We have
$$T_t=\int_0^11_{\mr_+}(W_u+\frac12\sqrt{t}(Hu+h_u))du,$$
and from this,
$$\left|T_t- \int_0^11_{\mr_+}(W_u+\frac12\sqrt{t}Hu)du\right|\le \int_0^1 1_{[-\sqrt t\vert h_u\vert, \sqrt t\vert h_u\vert]}(W_u+\frac12\sqrt t Hu)\, du.$$
When $\tau^t\leq 1$,  the process $X_v^t$ is bounded by $t^\alpha$ for $v\in [0,1]$, hence from \rf{drift} we have
$$|h_u|\le \int_0^u\left|b(X_v^t)-b(0)\right|dv\le Ct^\alpha$$
for some constant $C$ depending only on $b$. Therefore, upon writing 
$$Z_u = W_u+\frac12\sqrt t Hu,\qquad \phi = 1_{[-Ct^{1/2+\alpha}, \ Ct^{1/2+\alpha}]}$$
for simplicity we have 
\begin{equation}
\left|T_t- \int_0^11_{\mr_+}(W_u+\frac12\sqrt{t}Hu)du\right|\le \int_0^1 \phi(Z_s)\, ds.\label{firstapprox}
\end{equation}
In order to estimate the $p$th moment of this expression, where $p$ is assumed to be a positive integer without loss of generaility, we use 
$$\left(\int_0^1\phi(Z_s)\, ds\right)^p = p!\int_0^1\int_{u_1}^1\cdots\int_{u_{p-1}}^1\phi(Z_{u_1})\phi(Z_{u_2})\cdots \phi(Z_{u_p})\,du_1du_2\cdots du_p.$$
The process $Z$ is Markov. By conditioning at $u_p, u_{p-1}, \ldots, u_1$ successively we see that
$$\me_x\left(\int_0^1\phi(Z_s)\, ds\right)^p \le p!\max_{z\in\mr, \ 0\le u\le 1}\left|\me_z\int_0^u\phi(W_s+z)\,ds\right|^p.$$
However, $\me\phi(W_s+z)$ is the probability that $W_s$ lies in a certain interval of length $2Ct^{1/2+\alpha}$. Since the density of $W_s$ does not exceed $1/\sqrt{2\pi s}$, this probability is bounded by $2Ct^{1/2+\alpha}/\sqrt{2\pi s}$, hence
$$\me\int_0^u\phi(W_s+z)\, ds\le C_1t^{1/2+\alpha}$$
for another constant $C_1$. It follows from \rf{firstapprox} that
$$\me \left( \left\vert T_t- \int_0^11_{\mr_+}(W_u+\frac12\sqrt{t}Hu)du\right\vert^p; \tau^t\leq 1\right)\le Ct^{(1/2+\alpha)p}$$ 
and the proof is completed.
\end{proof}

We now state and prove the main result of this paper. 

\begin{theorem} \label{lpapprox} For any $p\geq1$, as $t\downarrow 0$ we have
\begin{align*}
\left\|T_t- \int_0^1 1_{\mr_+}(W_u)du-\frac12\sqrt{t}H\int_0^1u\,dL_u\right\|_p = O(t^{3/4}).
\end{align*}
\end{theorem}
\begin{proof} By the above lemma, we only need to focus on analyzing
$$\int_0^11_{\mr_+}(W_u+\frac12\sqrt{t}Hu)du.$$
Without loss of generality we assume that $H>0$. Let $\varPsi(x,s)=1_{\mr^+}(x+\sqrt{t}Hs/2)$. Since $\varPsi(x,s)$ is uniformly bounded and can be approximated pointwise by bounded continuous functions, we can apply the occupation time formula \rf{cop} for $\varPsi(x,s)$ to obtain
\begin{align*}
&\int_0^11_{\mr^+}(W_u+\frac12\sqrt{t}Hu)du-\int_0^11_{\mr^+}(W_u)du\\
=&\int_\mr\int_0^11_{\mr^+}(x+\frac12\sqrt{t}Hu)dL_u^x\,dx-\int_\mr\int_0^11_{\mr^+}(x)dL_u^x\,dx\\
=&\int_\mr\int_0^11_{(-\sqrt{t}Hu/2,\,0)}(x)dL_u^xdx\\
=&\int_\mr\int_0^11_{(-\sqrt{t}Hu/2,\,0)}(x)d(L_u^x-L_u^0)dx+\frac12\sqrt{t}H\int_0^1udL_u^0.
\end{align*}
Our theorem claims that the first term on the rightmost side of this chain of equalities is of the order $O(t^{3/4})$ in $L^p$. By the simple change of variable $x = \sqrt ty$ and integrating out the $u$ variable this term becomes 
\begin{align}
&\int_\mr\int_0^11_{(-\sqrt{t}Hu/2,\,0)}(x)d(L_u^x-L_u^0)dx\nonumber\\
=&\sqrt t \int_{-H}^0(L_{1}^{\sqrt{t}y}-L^0_{1})dy-\sqrt t\int_{-H}^0(L_{-2y/H}^{\sqrt{t}y}-L^0_{-2y/H})dy. \label{remainders}
\end{align}
We first deal with the second term. We have
\begin{align}
\me\left|\int_{-H}^0(L_{-y/H}^{\sqrt{t}y}-L^0_{-y/H})dy\right|^p
\le&C\int_{-H}^0\me\left|L_{-2y/H}^{\sqrt{t}y}-L^0_{-2y/H}\right|^pdy.\label{second term Lp 2}
\end{align}
Using the representation of $L^x_t$ in (\ref{local time representation}) we obtain for each $y\in(-H,0)$
\begin{align*}
&\me|L_{-y/H}^{\sqrt{t}y}-L^0_{-y/H}|^p\\
\leq & C\left\{|\sqrt{t}y|^p+\me\left|\int_0^{-y/H}(\mathrm{sgn}(W_s-\sqrt{t}y)-\mathrm{sgn}(W_s))dW_s\right|^p\right\}\\
\leq & C\left\{|\sqrt{t}|^pH^p+\me\left(\int_0^{-y/H}\left|\mathrm{sgn}(W_s-\sqrt{t}y)-\mathrm{sgn}(W_s)\right|^2ds\right)^{p/2}\right\}\\
=\, & C\left\{|\sqrt{t}|^pH^p+\me\left(4\int_0^{-y/H}1_{(-\sqrt ty, 0]}(W_s)ds\right)^{p/2}\right\}
\end{align*}
Note that we use $C$ to denote a general constant whose value may differ from one appearance to another. 
The expectation in the last expression can be estimated in the same way as in the proof of {\sc Lemma} \ref{cut tail} and we have
$$\me\left(\int_0^{-y/H}1_{(-\sqrt ty, 0]}(W_s)ds\right)^{p/2}\le C\left(\int_0^1\me1_{(-\sqrt t y, 0]}(W_s)\, ds\right)^{p/2}\le C t^{p/4}.$$
It follows that the $L^p$-norm of the second term in \rf{remainders} has the order $O(t^{3/4})$. The same conclusion holds for the first term in \rf{remainders} by a similar but easier analysis. The proof is thus completed.  
\end{proof}

\section{Concluding remarks}

The purpose of this paper is to show that the leading term of the geometric deviation from the classical occupation time arcsine law for a Riemannian Brownian motion has the form
$$\frac12\sqrt tH\int_0^1u\, dL_u.$$
We have shown an $L^p$ version of such a result. With some hard work an almost sure version is also possible:
$$\mp\lbr\lim_{t\downarrow0}\frac1{\sqrt t}\left|T_t - \int_0^11_{\mr_+}(W_s)\, ds - \frac12\sqrt tH\int_0^1u\, dL_u\right| = 0\rbr = 1.$$
The  amount of work involved for proving such a result is not commensurable with its reward; for this reason we do not include the proof of this result in this paper.

One can also attempt to interpret the approximation in the sense of distribution, i.e., describing the difference
$$\me\varphi(T_t) -\frac1\pi\int_0^1 \frac{\varphi(x)}{\sqrt{x(1-x)}}\, dx$$
for all smooth function $\varphi$ with compact support. In this respect, the result is a bit surprising and can be described as follows. Let $\mu_t$ be the
distribution of $T_t$. Then we have
$$\mu_t = \mu_0 + \frac13\sqrt{\frac1{2\pi}} H\mu_1 \sqrt t + O(t^{3/4}),$$
where $\mu_0$ is the standard arcsine distribution and $\mu_1 = \partial_x \delta_0$, the derivative of the Dirac delta function (in the sense of distribution). Here we see much less than in the $L^p$ approximation. In particular, the probabilistic role of the local time completely disappears from the picture, although it leaves its shadow in the computation
\begin{align*}
\me\int_0^1udL_u&=\me \left(L_1-\int_0^1L_udu\right)\\
&=\me\left(|W_1|-|W_1|\int_0^1u^{1/2}du\right)\\
&=\left(1-\int_0^1u^{1/2}du\right)\me|W_1|\\
&=2\cdot\frac13 \sqrt{\frac1{2\pi}}.
\end{align*}
Here we have used the fact that $L_u$ (the local time of Brownian motion $W$ at $x  = 0$) and $\vert W_u\vert$ have the same law for each fixed $u$.

\end{document}